\documentclass[11pt, twoside, leqno]{article}
\usepackage{amsmath,amsthm,amssymb,amscd,epsfig,esint}
\usepackage{verbatim}
\usepackage[english]{babel}
\usepackage{enumerate}

\allowdisplaybreaks 
\newtheorem*{thm*}{Theorem}
\newtheorem{thm}{Theorem}
\newtheorem{lem}[thm]{Lemma}
\newtheorem{prop}[thm]{Proposition}
\newtheorem{coro}[thm]{Corollary}

\newcommand{\deltabarra}{\bar{\partial}\,}

\def\R{\mathbb R}
\def\C{\mathbb C}

\def\cC{\mathcal C}

\def\B{\mathcal S}
\def\Re{\operatorname{Re}}

\def\Re{\text{Re}}
\def\B{\mathcal B}

\def\Id{\mathbf{Id}}

\title{Distributional solutions of the Beltrami equation}

\author{A. L. Bais\'on, A. Clop, J. Orobitg}

\date{}

\begin{document}

\maketitle

\abstract{We study the distributional solutions to the (generalized) Beltrami equation under Sobolev assumptions on the Beltrami coefficients.
In this setting, we prove that these distributional solutions are true quasiregular maps and they are smoother than expected, that is, 
they have second order derivatives in $L^{1+\varepsilon}_{loc}$, for some $\varepsilon >0$. }

\section{Introduction }

\noindent In this work we consider the $\R$-linear Beltrami equation
\begin{equation} \label{genBeltrami}
\deltabarra f(z) - \mu(z)\,\partial f(z) - \nu(z)\,\overline{\partial f(z)}=0, \hspace{1cm}a.e.\,\,z\in\C
\end{equation} 
where $\mu$ and $\nu$, called the Beltrami coefficients of $f$, are $L^\infty(\C;\C)$ functions such that $\| |\mu| + |\nu |\|_\infty=k<1$, 
which means that the equation \eqref{genBeltrami} is uniformly elliptic. The derivatives $\partial f$ and $\overline\partial f$ are understood in 
the distributional sense. Usually, when $\nu=0$, the equation \eqref{genBeltrami} is called the complex Beltrami equation. When $\mu=0$ it is called the conjugate
Beltrami equation. In the general form one says \eqref{genBeltrami} as the generalized Beltrami equation.

\noindent If a function $f:\Omega\to \C$ belonging to the Sobolev space $W^{1,2}_{loc}$ satisfies \eqref{genBeltrami}, 
then its distributional derivatives satisfy the \textit{distortion inequality},
\begin{equation}\label{distortionineq}
|\deltabarra f(z)|\leq k\,|\partial f(z)|\hspace{1cm}a.e.\,\,z\in\Omega.
\end{equation}
Such mappings $f$ are called $K$-quasiregular, $K=\frac{1+k}{1-k}$. Bijective $K$-quasiregular maps are called $K$-quasiconformal. 
Quasiregular maps are automatically H\"older continuous, and either constant or open and discrete. 
Weakly $K$-quasiregular maps are functions $f\in W^{1,p}_{loc}$ for some $p\geq 1$ and such that \eqref{distortionineq} is satisfied. 
In the work \cite{AIS} from Astala, Iwaniec and Saksman, it was shown that if a weakly $K$-quasiregular 
mapping belongs to $W^{1,p}_{loc}$ for some $p>1+k=\frac{2K}{K+1}$ actually it belongs to $W^{1,q}_{loc}$ for all $q< \frac{2K}{K-1}$ and, therefore, 
it is a true $K$-quasiregular map. The example $f(z)=(z|z|^{\frac1K-1})^{-1}$ 
shows that for any $p<\frac{2K}{K+1}$ there exists a weakly $K$-quasiregular map in $W^{1,p}_{loc}$ which is not $K$-quasiregular.
Later, Petermichl and Volberg \cite{PV} proved that the weaker assumption $p\geq \frac{2K}{K+1}$ 
is already sufficient for weakly $K$-quasiregular maps to be $K$-quasiregular. (See the monograph \cite{AIM} for a complete and detailed information).

\noindent If instead of looking at the distortion inequality \eqref{distortionineq} one looks at the regularity of Beltrami coefficients, 
one can say more about quasiregular maps.
This was already noticed by Iwaniec \cite{I} for the complex Beltrami equation (see \cite{Ko} or \cite{CC} for a proof for the generalized Beltrami equation).

\begin{thm}[\cite{I}]
Let $\mu, \, \nu\in \operatorname{VMO}$ be compactly supported, and such that $\| |\mu| + |\nu |\|_\infty=k<1$. If $f\in W^{1,r}_{loc}$ for some $r\in (1,\infty)$ and further $f$ 
solves \eqref{genBeltrami}
then $f\in W^{1,s}_{loc}$ for every $s\in(1,\infty)$. In particular, $f$ is quasiregular.
\end{thm}
\noindent In this sense, the work \cite{CMO} is a good reference to summarize this aspect about the regularity of the Beltrami coefficients.
 Recall that $\operatorname{VMO}$ denotes the space of vanishing mean oscillation functions, that is, the closure of compactly supported smooth functions with the
$\operatorname{BMO}$ norm. Obviously the above result is remarkable when $r$ is close to $1$. It turns out that one might go even beyond the space $W^{1,r}_{loc}$, $r>1$, 
and deal with solutions $f\in L^r_{loc}$ in the sense of distributions. Of course, this requires an extra degree of smoothness on $\mu$ and $\nu$, so that the 
distribution $\deltabarra f-\mu\,\partial f - \nu \,\overline{\partial f}$ is well defined. Such question for $\deltabarra f-\mu\,\partial f$ 
was treated in detail in \cite{CFMOZ}. Later, in \cite{B} interest arose in the counterparts for 
the distributional conjugate Beltrami operator $\partial f-\nu\,\overline{\partial f}$ with real valued coefficient $\nu\in W^{1,p}$, $p>2$.

%\section {Distributional solutions}
\noindent The well known Weyl's Lemma asserts that if $T$ is any (Schwartz) distribution such that $\deltabarra T(\varphi)=0$
for each test function $\varphi\in  \mathcal D$ (by $ \mathcal D$ we mean the algebra of compactly supported $C^{\infty}$ functions), 
then $T$ agrees with a holomorphic function.
In other words, distributional solutions to Cauchy-Riemann equation are actually strong solutions. 
When trying to extend this kind of result to the generalized Beltrami equation, one first must define the distribution
\begin{equation}\label{sense}
 \left(\overline{\partial}-\mu\partial -\nu \mathbf{C}\partial\right) T= \overline{\partial} T-\mu\partial T -\nu\overline{\partial T}\,.
\end{equation}
Here and henceforth, $\mathbf{C}$ denotes the complex conjugation operator $\mathbf{C} f = \overline{f}$.
The expression \eqref{sense} needs not make sense, because bounded functions in general do not multiply distributions nicely. 
However, if the multiplier is asked to exhibit some regularity, and the distribution $T$ is an integrable function,
then something may be done. Namely, given a function $f$ and a test function $\varphi$, one can write
\begin{equation}\label{defdist}
\begin{aligned}
\left\langle\,\left(\overline{\partial}-\mu\partial-\nu\mathbf{C}\partial \right) f\,,\,\varphi\right\rangle
& =  \left\langle\,\overline{\partial}f\,-\,\mu\,\partial f\,-\,\nu\,\overline{\partial f}\,,\,\varphi\,\right\rangle\,= \\
&= -\left\langle\,f\,,\,\partial\varphi\,\right\rangle\,+\,\left\langle\,f\,,\overline{\partial}\left(\,\overline{\mu}\,\varphi\,\right)\,\right\rangle\,
+\,\left\langle\,\overline{f}\,,\, \partial\left(\,\overline{\nu}\,\varphi\,\right)\,\right\rangle\,\\
&= -\left\langle\,f\,,\, \partial\varphi\,\right\rangle + \left\langle\,f\,,\varphi\overline{\partial \mu} \,\right\rangle + \left\langle\,\overline{f}\,,
\varphi\partial\overline{\nu} \,\right\rangle \\
& \qquad + \left\langle\,f\,,\,\overline{\mu}\overline{\partial}\varphi\,\right\rangle + 
\left\langle\,\overline{f}\,,\,\overline{\nu}\partial\varphi\,\right\rangle ,
\end{aligned}
\end{equation}
whenever each term makes sense. For instance, this is the case if  $\mu,\,\nu\in W^{1,p}(\C)$ are compactly supported
and $f\in L^q_{loc}(\C)$ with $\frac{1}{p}+\frac{1}{q}=1$. Hence we can call $\overline{\partial}f-\mu\partial f-\nu\overline{\partial f}$
the Beltrami distributional derivative of $f$.
Analogously, we say that $f$ is a distributional solution of
$$
\deltabarra f(z) = \mu(z)\,\partial f(z) + \nu(z)\,\overline{\partial f(z)} + h,
$$
when
$$
\left\langle\, \deltabarra f(z) - \mu(z)\,\partial f(z) - \nu(z)\,\overline{\partial f(z)} , \varphi \right\rangle\, = 
\left\langle\, h , \varphi \right\rangle    \qquad \text{for all } \varphi \in \mathcal D .
$$
Observe that the conditions $\mu,\,\nu\in W^{1,p}(\C)$ compactly supported, $h\in L^1_{loc}(\C)$ and $f\in L^{\frac{r}{r-1}}_{loc}(\C)$ with $r\le p$
are sufficient for \eqref{defdist} to make sense. 

\noindent Our first result extends  \cite[Corollary 3.3]{B} to the case $p=2$, and meets its $\C$-linear counterpart in \cite[Theorem 7]{CFMOZ}.

\begin{thm}\label{main}
 Let $1<r<p\le\infty$ and let $\mu, \nu \in W^{1,p}(\mathbb C)$ be compactly supported such that $\| |\mu| + |\nu| \|_{\infty}= k= \frac{K-1}{K+1}<1$. Assume that
 $f\in L^{\frac{r}{r-1}}_{loc}(\mathbb C)$ is a distributional solution of \eqref{genBeltrami}, that is,
 \begin{equation*} 
 \deltabarra f(z) =\mu(z)\,\partial f(z) + \nu(z)\,\overline{\partial f(z)}.
\end{equation*}
\begin{enumerate}
\item If $p>2$, then $f\in  W^{2,p}_{loc}(\mathbb C)$.
\item If $p=2$, then $f\in  W^{2,q}_{loc}(\mathbb C)$, for all $q<2$. 
\item When $\mu\equiv 0$, if $\frac{2K}{K+1}<p<2$ and $\frac{1}{r} > \frac{1}{p}+ \frac{K-1}{2K}$, then $f\in  W^{2,s}_{loc}(\mathbb C)$ for all 
$\frac{1}{s} > \frac{1}{p}+ \frac{K-1}{2K}$
\item When $K< 2$,  $K <p<2$ and $\frac{1}{r} > \frac{1}{p}+ \frac{K-1}{2K}$, then $f\in  W^{2,s}_{loc}(\mathbb C)$ for all 
$\frac{1}{s} > \frac{1}{p}+ \frac{K-1}{2K}$
\end{enumerate}
\end{thm}

\noindent In all previous cases, $f$ is quasiregular. We emphasize this fact in the following corollary. The result is new for cases $\nu\neq 0$ and $p\le 2$.
\begin{coro}
 Let $1< p\le\infty$ and let $\mu, \nu \in W^{1,p}(\mathbb C)$ be compactly supported such that $\| |\mu| + |\nu| \|_{\infty}= k<1$. Assume that
 $f\in W^{1,2}_{loc}(\mathbb C)$ satisfies
 \begin{equation*}
  \deltabarra f(z) =\mu(z)\,\partial f(z) + \nu(z)\,\overline{\partial f(z)}, \quad \text{a.e. }z\in\C .
\end{equation*}

\begin{enumerate}
\item If $p>2$, then $f\in  W^{2,p}_{loc}(\mathbb C)$.
\item If $p=2$, then $f\in  W^{2,q}_{loc}(\mathbb C)$, for all $q<2$.
\item When $\mu\equiv 0$ and $\frac{2K}{K+1}<p<2$, then $f\in  W^{2,s}_{loc}(\mathbb C)$ for all $\frac{1}{s} > \frac{1}{p}+ \frac{K-1}{2K}$.
\item When $K<2$ and  $K<p<2$, then $f\in  W^{2,s}_{loc}(\mathbb C)$ for all $\frac{1}{s} > \frac{1}{p}+ \frac{K-1}{2K}$.
\end{enumerate}
\end{coro}

\noindent We note that if $f$ is quasiconformal, following an argument similar to that in  \cite[Proposition 4]{CFMOZ}, we can remove the restriction
$K<2$. That is,

\begin{prop}\label{QC}
 Let $\mu, \nu \in W^{1,p}(\mathbb C)$ be compactly supported such that $\| |\mu| + |\nu| \|_{\infty}=\frac{K-1}{K+1}$ for some $K\ge 1$ 
 and $\frac{2K}{K+1}<p<2$. Assume that
 $f\in W^{1,2}_{loc}(\mathbb C)$ is a homeomorphism  satisfying \eqref{genBeltrami}. Then,
 $$
 f\in  W^{2,s}_{loc}(\mathbb C), \qquad \text{for all } \frac{1}{s} > \frac{1}{p}+ \frac{K-1}{2K} .
 $$
 \end{prop}
 
 \noindent At the end of this note we'll provide a proof. Let's mention that Proposition \ref{QC} is useful to study \eqref{genBeltrami} when 
 $\mu, \nu \in W^{s,p}(\mathbb C)$ with $1<s<2$ and $sp<2$ (see \cite{Pr}). 
 
 \noindent As stated in \cite[p.205]{CFMOZ}, the radial stretching $f(z)= z|z|^{\frac{1}{K}-1}$ if $|z|\le 1$ and $f(z)=z$ if $|z|>1$ shows the sharpeness of the 
above proposition because $f$ has Beltrami coefficient in $W^{1,p}$ for every $p<2$ and, however, the second derivatives of $f$ live in no better
space than $L^{\frac{2K}{2K-1}, \infty}$.

\section{Preliminaries}

Here we are considering the real inner product $\left\langle\,f\,,\,g\,\right\rangle\,=\,\Re\int\,f\overline{g}$.
We note that under this structure, the distributional derivatives behave as
$$\begin{aligned}
\left\langle\,\overline{\partial}f\,,\,\varphi\,\right\rangle\,=\,-\Re\int f\,\overline{\partial\varphi}\,&=\,-\,\left\langle\,f\,,\,\partial\varphi\,\right\rangle\\
\left\langle\,\partial f\,,\,\varphi\,\right\rangle\,=\,-\Re\int f\,\partial\overline{\varphi}\,&=\,-\,\left\langle\,f\,,\,\overline{\partial}\varphi\,\right\rangle\,.
\end{aligned}$$
Meanwhile the pointwise multiplication for proper functions $\mu$ satisfies
$$
\left\langle\,\mu\,f\,,\,\varphi\,\right\rangle\,=\,\Re\int\,\mu\,f\,\overline{\varphi}\,
=\,\Re\int\,f\,\overline{\overline{\mu}\,\varphi}\,=\,\left\langle\,f\,,\,\overline{\mu}\,\varphi\,\right\rangle\,.
$$
Naturally, this is consistent with the integral expression that we have used in the definition of the Beltrami distributional derivative.
The reason why we use this notion of duality is that the conjugation operator becomes $\mathbb R$-self-adjoint
$$
\left\langle\,\mathbf{C}f\,,\,g\,\right\rangle\,=\,\Re\,\int\,\overline{f}\,\overline{g}\,=
\,\Re\,\int\,f\,\,g\,=\,\left\langle\,f\,,\,\mathbf{C}g\,\right\rangle .
$$

\noindent Related to the Beltrami equations appear the Cauchy transform and the Beurling transform.
By $ \cC f$ one denotes the solid Cauchy transform of $f$,
$$
\cC f(z)=\frac{1}{\pi}\int_{\mathbb C}\frac{f(w)}{z-w}\,dA(w) ,
$$
while by $\B f$ one means the Beurling transform of the function $f$, a principal value convolution operator,
$$
\B f(z)=\frac{-1}{\pi}\lim_{\varepsilon\to 0}\int_{|w-z|\geq\varepsilon}\frac{f(w)}{(w-z)^2}\,dA(w) 
=\frac{-1}{\pi} \text{ pv}\int \frac{f(w)}{(w-z)^2}\,dA(w)  .
$$
The basic relations are $\partial\cC =\B $ and $\B \overline{\partial}= \partial$. By $\B^*$ we mean the singular integral operator obtained by simply conjugating the kernel of $\B$, that is,
$$
\B^* f(z)=\frac{-1}{\pi} \text{ pv}\int \frac{f(w)}{(\bar w-\bar z)^2}\,dA(w).
$$
As Calder\'on-Zygmund operators, $\B$ and $\B^*$ are bounded on $L^p(\C)$, $1<p<\infty$. Even more, they are bounded on $W^{s,p}(\C)$, because Sobolev spaces
$W^{s,p}=G_s*L^p$ are Bessel potential spaces.
Note that $\B^*$ is the $L^2$-inverse of $\B$ ($\B \B^* =\B^* \B = \Id$), moreover it is also the adjoint of $\B$ $\langle \B f, g\rangle = \langle f,\B^* g\rangle$
and  $\mathbf{C}\B(f) = \overline{\B (f)}=\B^* \mathbf{C}(f)$. From now, and abusing of notation, we'll denote $\mathbf{C}\B$ by $\overline{\B}$.
For compactly smooth functions $f$, we clearly have 
\begin{align*}
&\cC (\overline{\partial}f)= \overline{\partial}\cC (f)= f, \quad \partial\B(f) = \B(\partial f) 
\qquad\text{ and } \\
& \overline{\partial}\B (f)= \B (\overline{\partial}f) = \partial\cC (\overline{\partial}f) =\partial f, \quad
\partial \B^* (f)= \B^* (\partial f) = \overline{\partial}\cC^* (\partial f) =\overline{\partial} f, 
\end{align*}
and these identities extend for Sobolev functions in $W^{1,p}$, $p>1$.

\noindent When $h\in L^{p'}$, $p'=p/(p-1)$, $1<p<\infty$, the distributions $\partial h$ and $\overline{\partial} h$ act continuosly on $g\in W^{1,p}$
$$
\langle \partial h, g \rangle = -\langle h, \overline{\partial}g \rangle , \qquad \langle \overline{\partial} h, g \rangle = -\langle h,\partial g \rangle .
$$
In this setting, $\B(\partial h)$ can be defined as the distribution $\partial(\B h)$, that is,
$$
\langle \B (\partial h) ,    g \rangle  =\langle  \partial h ,   \B^* g \rangle  =-\langle  h ,  \overline{\partial} \B^* g \rangle =
-\langle  h , \B^* (\overline{\partial}g) \rangle =
-\langle \B h , \overline{\partial}g\rangle =
\langle \partial (\B h) , g\rangle
$$
 and also $\overline{\partial}\B h = \B\overline{\partial} h$. So, when $h\in L^{p'}$ we have $\B( \overline{\partial} h) = \partial h$
 and $\B^* (\partial h) =  \overline{\partial} h$ in the distributional sense.
 
\noindent After these considerations, we write
$$
\overline{\partial} f(z) - \mu(z)\,\partial f(z) - \nu(z)\,\overline{\partial f(z)}= h
$$
in the form
$$
(\Id-\mu\,\B -\nu\, \overline{\B}) (\overline{\partial} f) = h,
$$
where $\Id$ denotes the identity operator. Therefore, we realize that the study of \eqref{genBeltrami} is linked to the invertibility of the Beltrami 
operator $\Id-\mu\,\B -\nu\, \overline{\B}$ in an appropriate function space. 
This is the natural argument that one uses in the remarkable papers from Ahlfors \cite{Ahl}, 
Iwaniec \cite{I} and Astala, Iwaniec and Saksman \cite{AIS}, as well as the most recent works on the topic (see \cite{CMO}, \cite{Ko}, \cite{CC}).
In addition, in this note we will follow ideas from Clop et al. \cite{CFMOZ} and Baratchart et al. \cite{B}.

\section{Distributional solutions are true solutions}
We will begin proving that distributional solutions of \eqref{genBeltrami} under the hypothesis from Theorem \ref{main} have first derivatives locally
integrable. However, to get Theorem \ref{main} we don't use the Stoilow Theorem, which was fundamental in the approach of \cite{CFMOZ}.

\begin{thm}\label{firstDerivatives}
 Let $1<r<p\le\infty$ and let $\mu, \nu \in W^{1,p}(\mathbb C)$ be compactly supported such that $\| |\mu| + |\nu| \|_{\infty}= \frac{K-1}{K+1}$ for some $K\ge 1$.
 Fix $h\in L^{\frac{r}{r-1}}_{loc}(\mathbb C)$. 
 Assume that  $f\in L^{\frac{r}{r-1}}_{loc}(\mathbb C)$ is a distributional solution of 
 \begin{equation}\label{homo1}
 \deltabarra f(z) -\mu(z)\,\partial f(z) - \nu(z)\,\overline{\partial f(z)} = h.
\end{equation}
Then we have
\begin{enumerate}
\item $f\in  W^{1,\frac{r}{r-1}}_{loc}(\mathbb C)$, if $p\ge2$.
\item $f\in  W^{1,2}_{loc}(\mathbb C)$, if $\frac{2K}{K+1}<p<2$ and $\frac{1}{r} > \frac{1}{p}+ \frac{K-1}{2K}$.
\end{enumerate}
\end{thm}

\begin{proof}
The scheme of the proof is as follows. First of all we localize the solution $f$ to get another generalized Beltrami equation where all the involved
terms will have compact support. Then we will check that a bounded combination of this located function have first derivatives. Finally,
we transfer the integrability of the first derivatives to the original solution $f$.

\noindent Consider an smooth function with compact support $\psi\in C^\infty_c\left(\C\right)$ with real values and define
$F\,=\,\psi\,f \in L^\frac{r}{r-1}_c\left(\C\right)$. Obviously, the global regularity on $F$ will imply the local regularity on $f$, as we wish.

\noindent Using \eqref{homo1}, $\mu, \nu\in W_c^{1,p}(\mathbb C)$ and $\psi$ is real valued, one easly chekcs that $F$ is a distributional solution of
\begin{equation}\label{eq para F normalizada}
\overline{\partial}F\,=\,\mu\,\partial F\, + \,\nu\,\overline{\partial F}\,+\,H\,,
\end{equation}
where $H =\,\left(\,f\,-\,\nu\,\overline{f}\,\right)\, \overline{\partial} \psi\, -\, \mu\, f \partial\psi\, +\, h\psi $
belongs to $L^\frac{r}{r-1}_c\left(\C\right)$.

\noindent Now we define $G\,=\,\left(\,\Id\,-\,\mu\,\mathcal{B}\,-\,\nu\,\mathbf{C}\,\right)\,F$. The distributional derivative $\overline{\partial}G$ of $G$ 
has sense because $G\in L^{\frac{r}{r-1}}_c\left(\C\right)$
(the compact support of $G$ comes from the compact support of $F$, $\mu$ and $\nu$). Then for any $\varphi\in C^\infty_c\left(\C\right)$ we have 
$$\aligned
\left\langle\,\overline{\partial} G\,,\,\varphi \,\right\rangle\,&=\,-\,\left\langle\,G \,,\,\partial\varphi \,\right\rangle\,=\,-\,\left\langle\,F\,
-\,\nu\,\overline{F}\,-\,\mu\, \mathcal{B}F \,,\,\partial\varphi \,\right\rangle\\
&=\,-\,\left\langle\,F \,,\,\partial\varphi \,\right\rangle\,+\,\left\langle\,\nu\,\overline{F} \,,\,\partial\varphi \,\right\rangle\,
+\,\left\langle\,\mu\,\mathcal{B}F \,,\,\partial\varphi \,\right\rangle\\
&=\,\left\langle\,\overline{\partial}F \,,\,\varphi \,\right\rangle\,-\,\left\langle\,\overline{\partial F} \,,\,\overline{\nu}\,\varphi \,\right\rangle\,
-\,\left\langle\,\overline{F}\,\overline{\partial}\nu \,,\,\varphi \,\right\rangle\,+\,\left\langle\,\mathcal{B}F \,,\,
\partial\left(\,\overline{\mu}\,\varphi\,\right) \,\right\rangle\,-\,\left\langle\,\overline{\partial}\mu\,\mathcal{B}F\,,\,\varphi\,\right\rangle\\
&=\,\left\langle\,\overline{\partial}F \,,\,\varphi \,\right\rangle\,-\,\left\langle\,\overline{\partial F} \,,\,\overline{\nu}\,\varphi \,\right\rangle\,
-\,\left\langle\,\overline{F}\,\overline{\partial}\nu \,,\,\varphi \,\right\rangle\,
-\,\left\langle\,\overline{\partial}\partial\mathcal{C} F \,,\,\overline{\mu}\,\varphi\,\right\rangle\,
-\,\left\langle\,\overline{\partial}\mu\,\mathcal{B}F\,,\,\varphi\,\right\rangle\\
&=\,\left\langle\,\overline{\partial}F \,,\,\varphi \,\right\rangle\,-\,\left\langle\,\overline{\partial F} \,,\,\overline{\nu}\,\varphi \,\right\rangle\,
-\,\left\langle\,\overline{F}\,\overline{\partial}\nu \,,\,\varphi \,\right\rangle\,
-\,\left\langle\,\partial F \,,\,\overline{\mu}\,\varphi\,\right\rangle\,-\,\left\langle\,\overline{\partial}\mu\,\mathcal{B}F\,,\,\varphi\,\right\rangle\\
&=\,\left\langle\,H\,-\,\overline{F}\,\overline{\partial}\nu\,-\,\left(\,\mathcal{B}F\,\right)\,\overline{\partial}\mu\,,\,\varphi\,\right\rangle.
\endaligned$$
Here we have used $\overline{\partial}\mathcal{C}F=F$ in the fifth line and \eqref{eq para F normalizada} in the last one. 
Consequently $G$ is a distributional solution of 
\begin{equation}\label{G es sol gener L->W}
\overline{\partial}G\,=\,H\,-\,\overline{F}\,\overline{\partial}\nu\,-\,\left(\,\mathcal{B}F\,\right)\,\overline{\partial}\mu\,
\end{equation}
and clearly $\overline{\partial}G\in L^s_{c}\left(\C\right)$ with $s=\frac{pr'}{p+r'}$. 
Applying the Cauchy transform to the equation \eqref{G es sol gener L->W} we get (because $G$ has compact support)
$$
G\,= \mathcal{C}\left(\,H\,-\,\overline{F}\,\overline{\partial}\nu\,-\,\left(\,\mathcal{B}F\,\right)\,\overline{\partial}\mu\,\right)
$$
and then
$$ 
\partial G\,=\,\mathcal{B}\left(\,H\,-\,\overline{F}\,\overline{\partial}\nu\,-\,\left(\,\mathcal{B}F\,\right)\,
\overline{\partial}\mu\,\right)\in L^s\left(\,\C\,\right)\,.
$$
Therefore, $G\in W^{1,s}_c\left(\,\C\,\right)$. 

\vspace{0,3cm}

\noindent \textbf{Remark.} If $\mu \equiv 0$ then $G= (\Id -\nu \,\mathbf{C})F$, that is, $F=\dfrac{G-\nu\overline{G}}{1-|\nu|^2} \in W^{1,s}_c\left(\,\C\,\right)$.
This reinforces the fact already observed in \cite{B} that the study of the regularity in the conjugate Beltrami equation 
is easier than in the complex Beltrami equation.

\vspace{0,3cm}

\noindent We want to transfer the regularity from $G$ to $F$, and therefore to $f$.
We'll work with some identities with the  distributional derivatives of $F$ and finally we'll conclude that $\partial F$ is an integrable function.
For all $\varphi\in C^\infty_c\left(\,\C\,\right)$ we have
$$\begin{aligned}
\left\langle\,\partial G\,,\,\varphi \,\right\rangle\,&=\,-\,\left\langle\,G \,,\,\overline{\partial}\varphi \,\right\rangle\,=\,\left\langle\,-\,F\,
+\,\nu\,\overline{F}\,+\,\mu\,\left(\,\mathcal{B}F\,\right)\,,\,\overline{\partial}\varphi\,\right\rangle\\
&=\,\left\langle\,\partial F \,,\,\varphi \,\right\rangle\,
+\,\left\langle\,\overline{F} \,,\,\overline{\partial}\left(\,\overline{\nu}\,\varphi\,\right) \,\right\rangle\,
+\,\left\langle\,\mathcal{B}F \,,\,\overline{\partial}\left(\,\overline{\mu}\,\varphi\,\right) \,\right\rangle\,
-\,\left\langle\,\overline{F}\,\partial\nu \,,\,\varphi \,\right\rangle\,
-\,\left\langle\,\left(\,\mathcal{B}F\,\right)\,\partial\mu \,,\,\varphi \,\right\rangle .\\
\end{aligned}$$
Now, we replace $\langle \,\overline{F}\,,\,\overline{\partial}\left(\,\overline{\nu}\,\varphi\,\right)\,\rangle$ 
using that $\partial\overline{F}=\overline{\overline{\partial}F}$,  moreover $\overline{\partial}F$ satisfies \eqref{eq para F normalizada} 
and the expression 
$$
\langle\partial\, \overline{F}\,,\,\phi \,\rangle \,= \,\langle\,\overline{\partial F}\,,\,\mu\,\phi\,\rangle\, 
+\, \langle \,\partial F\,,\,\nu\,\phi\,\rangle \,+\, \langle\, \overline{H}\,,\,\phi\,\rangle
$$
holds for any $\phi\in W^{1,p}_c\left(\C\right)\cap L^\infty\left(\C\right)$. After this change we reach that
$$\begin{aligned}
\left\langle\,\partial G\,,\,\varphi \,\right\rangle\,
&=\,\left\langle\,\partial F \,,\,\varphi \,\right\rangle\,-\,\left\langle\,\partial F \,,\,\left|\,\nu\,\right|^2 \,\varphi\,\right\rangle\,-\,\left\langle\,\overline{\partial F} \,,\,\mu\,\overline{\nu}\,\varphi \,\right\rangle\,+\,\left\langle\,\mathcal{B} F \,,\,\overline{\partial}\left(\,\overline{\mu}\,\varphi\,\right) \,\right\rangle\\
&\hspace{3cm}\,- \,\left\langle\,\nu\,\overline{H} \,,\,\varphi \,\right\rangle\,-\,\left\langle\,\overline{F}\,\partial\nu\,+\,\left(\,\mathcal{B}F\,\right)\,\partial\mu \,,\,\varphi \,\right\rangle .
\end{aligned}$$
Simplifying and rearranging terms,
$$
\begin{aligned}
& \left\langle \partial \,F\,,\, \left(\,1\,-\,|\nu|^2\,\right)\,\varphi\,\right\rangle \,
+\,\left\langle \,\mathcal{B} F\,,\,\overline{\partial} \left(\,\overline{\mu}\,\varphi\,\right)\,\right\rangle\,
-\,\left\langle \,\overline{\partial F}\,,\, \,\mu\,\overline{\nu}\,\varphi\,\right\rangle\,= \\
& = \left\langle\,\partial G\,
-\,\nu\,\overline{H}\,-\,\overline{F}\,\partial\nu\,-\,\left(\,\mathcal{B}F\,\right)\,\partial\mu\,,\,\varphi\,\right\rangle.
\end{aligned}
$$
Furthermore, this equality extends to any $\varphi$ in $W^{1,p}_c\left(\C\right)\cap L^\infty\left(\C\right)$. 
In particular, for functions of type $\varphi=\frac{\Psi}{1-|\nu|^2}$ with $\Psi\in C^\infty_c\left(\C\right)$ arbitrary. 
Thus,
\begin{multline*}
 \left\langle \partial\, F\,,\, \Psi\,\right\rangle \,-\,\left\langle \,\partial\,\mathcal{B}  F\,,\,\frac{\overline{\mu}}{1-|\nu|^2}\Psi\,\right\rangle\,
-\,\left\langle \,\overline{\partial F}\,,\, \frac{\mu\overline{\nu}}{1-|\nu|^2}\,\Psi\,\right\rangle\, \\
=\,\left\langle\,\frac{\partial G\,-\,\nu\,\overline{H}\,-\,\overline{F}\,\partial\nu\,
-\,\left(\,\mathcal{B}F\,\right)\,\partial\mu}{\,1\,-\,\left|\,\nu\,\right|^2}\,,\,\Psi\,\right\rangle\,.
\end{multline*}
Recall that $\partial\B F= \B\partial F$ and $\B^*$ is the adjoint operator of $\B$. Rewriting the above equality we get
\begin{multline}\label{segunda aprox distrib BeG}
\left\langle\,\partial F\,,\,\left(\,\Id\,-\,\mathcal{B^*}\,\frac{\overline{\mu}}{\,1\,-\,\left|\,\nu\,\right|^2}\,-\,\mathbf{C}\,
\frac{\mu\,\overline{\nu}}{\,1\,-\,\left|\,\nu\,\right|^2}\,\right)\,\Psi\,\right\rangle\, \\
=\,\left\langle\,\frac{\partial G\,-\,\nu\,\overline{H}\,-\,\overline{F}\,\partial\nu\,-\,
\left(\,\mathcal{B}F\,\right)\,\partial\mu}{\,1\,-\,\left|\,\nu\,\right|^2}\,,\,\Psi\,\right\rangle\,.
\end{multline}
Now the invertibility of the operator
$$
T:= \,\Id\,-\,\mathcal{B^*}\,\frac{\overline{\mu}}{\,1\,-\,\left|\,\nu\,\right|^2}\,-\,\mathbf{C}\,
\frac{\mu\,\overline{\nu}}{\,1\,-\,\left|\,\nu\,\right|^2}
$$
and its adjoint
$$
T^*:= \,\Id\,-\,\frac{\mu}{\,1\,-\,\left|\,\nu\,\right|^2}\, \mathcal{B}\,
-\, \frac{\overline{\mu}\,\nu}{\,1\,-\,\left|\,\nu\,\right|^2}\, \mathbf{C}
$$
come into play (see Lemma \ref{invert} below).

\noindent Second part of Theorem \ref{firstDerivatives}.
Now, since $\frac{2K}{K+1}<p<2$  and  $\frac{1}{r}>\frac{1}{p}+\frac{K-1}{2K}$ we have
$$\frac{2K}{K+1}<s=\frac{pr'}{p+r'}<2<s'<\frac{2K}{K-1}\,.
$$
Thus, by Lemma \ref{invert}, $T$ and $T^*$ are invertible on $L^s(\mathbb C)$ and on $L^{s'}(\mathbb C)$.
The right hand side of \eqref{segunda aprox distrib BeG} extends by duality to $\Psi\in L^{s'}\left(\C\right)$. 
Then we use as test function $\Psi:= T^{-1}\, \varphi$ with $\varphi\in C^\infty_c\left(\C\right)$  arbitrary. In this case we obtain
$$
\left\langle\,\partial F\,,\,\varphi\,\right\rangle=
\left\langle\,\frac{\partial G\,-\,\nu\,\overline{h}\,-\,\overline{F}\,\partial\nu\,-\,\left(\,\mathcal{B}F\,\right)
\,\partial\mu}{\,1\,-\,\left|\,\nu\,\right|^2}\,, \,T^{-1}\varphi\,\right\rangle\,,
$$
or equivalently, using the adjoint operator,
$$
\left\langle\,\partial F\,,\,\varphi\,\right\rangle=
\left\langle\,( T^*)^{-1}\,\left(\,\frac{\partial G\,-\,\nu\,\overline{h}\,-\,\overline{F}\,\partial\nu\,-\,
\left(\,\mathcal{B}F\,\right)\,\partial\mu}{\,1\,-\,\left|\,\nu\,\right|^2}\,\right)\,,\,\varphi\,\right\rangle\,.
$$
This equality means that $\partial F$ belongs to  $L^s\left(\,\C\,\right)$. 
Therefore, thanks to \eqref{eq para F normalizada}, we have $F\in W^{1,s}_c\left(\C\right)$. 
Once $F$ has first derivatives in $ L ^ s (\ C) $ for some $ s> 1 $, we can ensure that
$\mathcal{B}\overline{\partial}F\,=\,\partial F$ and write
\begin{equation}\label{FiH}
 \left(\,\Id\,-\,\mu\,\mathcal{B}\,\,-\,\nu\,\overline{\mathcal{B}}\,\right)\,\overline{\partial} F \,=\,H\, ,
 \quad \text{ with $H\in L^t_c\left(\C\right)$ for all $t\leq\frac{r}{r-1}$.}
\end{equation}
Since $r<p<2$ one has $\dfrac{r}{r-1}>2$. For $t=2$ we use the invertibility of the Beltrami operator,
$$
\overline{\partial} F \,=\,\left(\,\Id\,-\,\mu\,\mathcal{B}\,\,-\,\nu\,\overline{\mathcal{B}}\,\right)^{-1}\,H\,
$$
to get $F\in W^{1,2}_{c}(\C)$, and so  $f\in W^{1,2}_{loc}(\C)$ as we wished.

\vspace{0,3cm}

\noindent First part of Theorem \ref{firstDerivatives}. In this case the operators $T$ and $T^*$ are invertible on $L^q (\C)$ for all $q\in (1,\infty)$.
Repeating the previous argument we arrive to \eqref{FiH} and then
$$
\overline{\partial} F \,=\,\left(\,\Id\,-\,\mu\,\mathcal{B}\,\,-\,\nu\,\overline{\mathcal{B}}\,\right)^{-1}H\, \in 
L^\frac{r}{r-1}(\C) ,
$$
implying that $f\in W^{1,\frac{r}{r-1}}_{loc}\left(\C\right)$.

\end{proof}

\begin{lem}\label{invert}
Let $\mu, \nu \in L^{\infty}(\mathbb C)$ be compactly supported such that $\| |\mu| + |\nu| \|_{\infty}= k<1$.
Let $T$ and $T^*$ the two operators  previously defined.
\begin{enumerate}[(a)]
\item Both operators are invertible on $L^t (\mathbb C)$ when $1+k=\frac{2K}{K+1}<t<\frac{2K}{K-1}=1+1/k$.
 \item If in addition, $\mu, \nu \in \operatorname{VMO}(\mathbb C)$ then both operators are invertible on $L^t (\mathbb C)$
 for any $t\in (1,\infty)$.
\end{enumerate}
\end{lem}

\begin{proof}

 We only proceed with the operator $T^*$, because the argument for $T$ is the same. Define
 $$
 \mu_1\,:=\,\frac{\mu}{1-\left|\nu\right|^2}\hspace{1cm}\text{y}\hspace{1cm}\nu_1\,:=\,\frac{\nu\,\overline{\mu}}{1-\left|\nu\right|^2}\,.
$$
Clearly $\| |\mu_1| + |\nu_1| \|_{\infty}= \| \frac{|\mu|}{1-|\nu |}\|_{\infty} \le k$ because $\| |\mu| + |\nu| \|_{\infty}= k<1$. Then
\begin{align*}
T^*= \left(\, \Id\,-\,\mu_1\,\mathcal{B} -\,\nu_1\,\mathbf{C}\,\right)\,  & =\,\left(\,\Id\,-\,\nu_1\,\mathbf{C}\,\right) 
\left(\,\Id\,-\,\frac{\mu_1}{1-\left|\nu_1\right|^2}\,\mathcal{B}\,-\,\frac{\nu_1\,\overline{\mu_1}}{1-\left|\nu_1\right|^2}\,\mathbf{C}{\mathcal{B}}\,\right) \\
& =\,\left(\,\Id\,-\,\nu_1\,\mathbf{C}\,\right)
\left(\,\Id\,-\,\mu_2\,\mathcal{B}\,-\,\nu_2\,\mathbf{C}{\mathcal{B}}\,\right),
\end{align*}
where the last equality defines $\mu_2$ and $\nu_2$.
Since the operator $\left (\Id - \nu_1 \, \mathbf {C} \right) $ is invertible in all spaces 
$L^t\left(\C\right)$, $1\leq t\leq\infty$,
the invertibility of $T^*$ is reduced to the invertibility of the Beltrami operator 
$(\Id\,-\,\mu_2\,\mathcal{B}\,-\,\nu_2\,\mathbf{C}{\mathcal{B}})$, 
which follows from the outcome of \cite{AIS}.

\noindent When $\mu ,\nu \in \operatorname{VMO}\cap L^{\infty}_c$ with $\| |\mu| + |\nu| \|_{\infty}= k<1$, it is an easy calculation to check that
$\mu_1 ,\nu_1 \in \operatorname{VMO}\cap L^{\infty}_c$ and so also $\mu_2$ and $\nu_2$. Therefore the invertibility of 
$(\Id\,-\,\mu_2\,\mathcal{B}\,-\,\nu_2\,\mathbf{C}{\mathcal{B}})$ on $L^t (\mathbb C)$ for any $t\in (1,\infty)$
follows from \cite{I}.
 
\end{proof}

\section{Proof of Theorem \ref{main}}

Consider $\psi\in C^\infty_c\left(\C\right)$ with real values and define
$F\,=\,\psi\,f \in L^\frac{r}{r-1}_c\left(\C\right)$.
Our goal is to show that $F$ has second derivatives in some Lebesgue space,
and then $f$ will have the same local regularity.
The method consists in applying several times the invertibility of the generalized Beltrami operators.
For Parts 1 and 2 we rely on the invertibility of the Beltrami operator on some Sobolev spaces.
On the other hand, in Part 3 we can only use this invertibility in Lebesgue spaces.
We'll prove each part separately.

\noindent Part 1. Fix $p>2$. By Theorem \ref{firstDerivatives} we know  $F\in W^{1,\frac{r}{r-1}}_c\left(\C\right)$ 
and satifies  \eqref{eq para F normalizada}
$$
\overline{\partial}F\,=\,\mu\,\partial F\, + \,\nu\,\overline{\partial F}\,+\,H\,,
$$
where 
\begin{equation}\label{defH}
 H =\,\left(\,f\,-\,\nu\,\overline{f}\,\right)\, \overline{\partial} \psi\, -\, \mu\, f \,\partial\psi
\end{equation}
belongs to $L^\frac{r}{r-1}_c\left(\C\right)$.
Following this scheme we also get $f\in W^{1,\frac{r}{r-1}}_{loc}\left(\C\right)$. 
Therefore, by the Sobolev embedding, we have $f\in L^\frac{2r}{r-2}_{loc}\left(\C\right)$ if $r>2$, $f\in L^q_{loc}\left(\C\right)$ 
for all $q<\infty$ if $r=2$ and $f\in L^\infty_{loc}\left(\C\right)$ if $r<2$. Consequently, $H\in L^\frac{2r}{r-2}_c\left(\C\right)$ when $r>2$, 
$H\in L^q_c\left(\C\right)$ for all $q<\infty$ if $r=2$ and $H\in L^\infty_c\left(\C\right)$ when $r<2$. 
From the identity
\begin{equation}\label{recursiu}
\overline{\partial}F\,=\,\left(\,\Id\,-\mu\,\mathcal{B}\,-\,\nu\,\overline{\mathcal{B}}\,\right)^{-1}\,H\,
\end{equation}
and by the invertibility of the Beltrami operator we have $F\in W^{1,t}_c\left(\C\right)$ for some $t>2$ regardless of the value of $r$. 
Then, $F\in L^\infty_{c}\left(\C\right)$, $f\in L^\infty_{loc}\left(\C\right)$ and so $H\in L^\infty_c\left(\C\right)$. 
Re-using the invertibility of the operator, this time we get $F\in W^{1,q}_c\left(\C\right)$ for all $q<\infty$. 
Immediately we have $f\in W^{1,q}_{loc}\left(\C\right)$ for all $q<\infty$. With this, it is easy to verify that
$H\in W^{1,p}_c\left(\C\right)$. 

\noindent Now, since $p>2$, we know (see \cite[Section 3]{CMO} and the argument in \cite[Lemma 10]{CC}) that the generalized Beltrami operator
$$
\Id\,-\,\mu\,\mathcal{B}\,-\,\nu\,\overline{\mathcal{B}}:\,W^{1,p}\left(\,\C\,\right)\,\to\,W^{1,p}\left(\,C\,\right)
$$
is bounded and boundedly invertible. From \eqref{recursiu} $\overline{\partial}F \in W^{1,p}_c\left(\C\right)$
(and $\partial F = \B (\overline{\partial} F )$ too), 
that is, $f\in W^{2,p}_{loc}\left(\C\right)$.\\

\noindent Part 2. When $p=2$ the proof follows the same argument as in Part 1. Just keep in mind that in this case (see Lemma \ref{inverL2} below)
the operator
$$
\Id\,-\,\mu\,\mathcal{B}\,-\,\nu\,\overline{\mathcal{B}}:\,W^{1,q}\left(\,\C\,\right)\,\to\,W^{1,q}\left(\,C\,\right)
$$
is bounded and boundedly invertible for all $q\in(1,2)$ (but not for $q=2$). Thus, $f\in W^{2,q}_{loc}\left(\C\right)$ for all $q<2$ as we claimed.\\

\noindent For a moment, we consider Part 3 and Part 4 together. Recall that $\frac{2K}{K+1}<p<2$ and $\frac{1}{r} > \frac{1}{p}+ \frac{K-1}{2K}$. 
By Theorem \ref{firstDerivatives} we know  $F\in W^{1,2}_c\left(\C\right)$ and
$f\in W^{1,2}_{loc}\left(\C\right)$. Then, the function $H$ given by \eqref{defH} belongs to $L^t_{c}(\C)$ for all $t<\infty$.
The invertibility of $\Id\,-\,\mu\,\mathcal{B}\,-\,\nu\,\overline{\mathcal{B}}$ and \eqref{recursiu} give $\overline{\partial} F\in L^q_c(\C)$
(and so, $\partial F$ too) for all $2\le q<\frac{2K}{K-1}$. In particular we obtain $f\in L^{\infty}_{loc}(\C)$ and with that we achieve 
$H\in W^{1,p}_c(\C)\cap L^{\infty}$.

\noindent Define $G:=\partial F$. Our goal is to check that the distribution $\overline{\partial} G$ is in fact an integrable function.
Consider $\Psi\in C^\infty_c\left(\C\right)$ arbitrary, then it is true that
$$
\begin{aligned}
\left\langle\, \overline{\partial} G\,,\,\Psi \,\right\rangle\,&=\,-\left\langle\,G \,,\,\partial\Psi \,\right\rangle\,
=\,-\left\langle\,\partial F \,,\, \partial\Psi\,\right\rangle\, =\,-\left\langle\,\B(\overline{\partial} F) \,,\, \partial\Psi\,\right\rangle\, \\
&  =\,-\left\langle\,\overline{\partial} F \,,\, \B^*(\partial\Psi)\,\right\rangle\,
=\,-\left\langle\, \overline{\partial}F\,,\,\overline{\partial} \Psi \,\right\rangle
\\
&= -\left\langle\,\partial F \,,\,\overline{\mu}\,\overline{\partial} \Psi \,\right\rangle\,-\,\left\langle\,\overline{\partial F} \,,\,\overline{\nu}\,\overline{\partial} \Psi \,\right\rangle\,-\,\left\langle\,H \,,\,\overline{\partial} \Psi \,\right\rangle\\
&=-\left\langle\,G \,,\,\overline{\mu}\,\overline{\partial} \Psi \,\right\rangle\,-\,\left\langle\,\overline{G} \,,\,\overline{\nu}\,\overline{\partial} \Psi \,\right\rangle\,-\,\left\langle\,H \,,\,\overline{\partial} \Psi \,\right\rangle\\
&=\left\langle\,\partial G \,,\,\overline{\mu}\, \Psi \,\right\rangle\,+\,\left\langle\,\partial\overline{G} \,,\,\overline{\nu} \,\Psi \,\right\rangle\,+\,\left\langle\,\partial H \,,\,\Psi \,\right\rangle\\
&\hspace{2,6cm}+\,\left\langle\,G\,\partial\mu\,+\,\overline{G}\,\partial\nu\,,\,\Psi\,\right\rangle .\\
\end{aligned}
$$
That is,
\begin{equation}\label{ni idea 1}
\left\langle\, \overline{\partial} G\,,\,\Psi \,\right\rangle\,=\left\langle\,\partial G \,,\,\overline{\mu}\, \Psi \,\right\rangle\,+\,\left\langle\,\partial\overline{G} \,,\,\overline{\nu} \,\Psi \,\right\rangle\,+\,\left\langle\,G\,\partial\mu\,+\,\overline{G}\,\partial\nu\,+\,\partial H \,,\,\Psi \,\right\rangle
\end{equation}
Note that $\overline{\mu}\Psi,\,\overline{\nu}\Psi\in W^{1,p}_c\left(\C\right)$. Thanks to this we can ensure that
\begin{equation}\label{ni idea 2}
\left\langle\,\partial G\,,\,\overline{\mu}\,\Psi\,\right\rangle \,
= \,-\,\left\langle\, G\,,\,\overline{\partial}\left(\,\overline{\mu}\,\Psi\,\right)\,\right\rangle\,
=\,-\left\langle\, G\,,\,\partial\mathcal{B}^*\left(\,\overline{\mu}\,\Psi\,\right)\,\right\rangle\,
=\,\left\langle\,\overline{\partial}G\,,\,\mathcal{B}^*\left(\,\overline{\mu}\,\Psi\,\right)\,\right\rangle\,
\end{equation}
and
\begin{equation}\label{ni idea 3}
\,\left\langle\,\partial\overline{G} \,,\,\overline{\nu} \,\Psi \,\right\rangle\,=\,\left\langle\,\mathbf{C}\overline{\partial}G \,,\,\overline{\nu} \,\Psi \,\right\rangle\,=\,\left\langle\,\overline{\partial}G \,,\,\mathbf{C}\left(\,\overline{\nu} \,\Psi\,\right) \,\right\rangle\,.
\end{equation}
Plugging \eqref{ni idea 2} and \eqref{ni idea 3} into \eqref{ni idea 1}, we have
\begin{equation}\label{prelaplacian}
\left\langle\,\overline{\partial} G \,,\,\left(\,\Id\,-\,\mathcal{B^*}\overline{\mu} \,-\,\mathbf{C}\overline{\nu}\,\right)\,\Psi \,\right\rangle\,=\,\left\langle\, G\,\partial \mu\,+\,\overline{G}\,\partial\nu\,+\,\partial H \,,\,\Psi \,\right\rangle\,.
\end{equation} 
By construction, $G\,\partial\mu$, $\overline{G}\,\partial\nu$, $\partial H\in L^s_c\left(\C\right)$ for any $1>\frac{1}{s}>\frac{1}{p}+\frac{K-1}{2K}$. 
Therefore, \eqref{prelaplacian} has sense for all $\Psi\in L^{s'}\left(\C\right)$.

\noindent Part 3. Since $\mu\equiv 0$, \eqref{prelaplacian} becomes
\begin{equation*}\label{prelaplacian2}
\left\langle\,\overline{\partial} G \,,\,\left(\,\Id\,-\,\mathbf{C}\overline{\nu}\,\right)\,\Psi \,\right\rangle\,
=\,\left\langle\, \overline{G}\,\partial\nu\,+\,\partial H \,,\,\Psi \,\right\rangle\,.
\end{equation*} 
The operator $\Id\,-\,\mathbf{C}\overline{\nu} $ and its adjoint $\Id\,-\,\nu \mathbf{C}$ are obviously invertible on $L^t$ for any $t$. Take
$\Psi:= (\Id\,-\,\mathbf{C}\overline{\nu})^{-1}\varphi=\frac{\varphi +\nu\bar\varphi}{1-|\nu|^2}$ and so
\begin{equation*}
\left\langle\,\overline{\partial} G \, , \varphi \,\right\rangle\,
=\,\left\langle\, \overline{G}\,\partial\nu\,+\,\partial H \,,\, (\Id\,-\,\mathbf{C}\overline{\nu})^{-1}\varphi \,\right\rangle\,
=\,\left\langle\, (\Id\,-\,\nu \mathbf{C})^{-1}( \overline{G}\,\partial\nu\,+\,\partial H) \,,\, \varphi \,\right\rangle\,
\end{equation*} 
for all $\varphi\in L^{s'}$. In conclusion $\bar\partial \partial F=\overline{\partial} G\in L^{s}_c$ and
$f\in W^{2,s}_{loc}\left(\C\right)$ if $\frac{1}{s}>\frac{1}{p}+\frac{K-1}{2K}$.

\noindent Part 4. The operator $\Id\,-\,\mathcal{B^*}\overline{\mu} \,-\,\mathbf{C}\overline{\nu}$ and its adjoint 
$\Id\,-\,\mu \mathcal{B} \,-\,\nu\mathbf{C}$ are invertible on $L^t$ if $\frac{2K}{K+1}<t< \frac{2K}{K-1}$. This is why we have restricted the case to $K< 2$
and $K<p<2$. Taking $\Psi:= (\Id\,- \mathcal{B^*}\overline{\mu} -\,\mathbf{C}\overline{\nu})^{-1}\varphi$, \eqref{prelaplacian} becomes
$$
\begin{aligned}
\left\langle\,\overline{\partial} G \,, \varphi \right\rangle\,
& = \left\langle\, G\,\partial \mu\,+\,\overline{G}\,\partial\nu\,+\,\partial H \,,\,\left(\,\Id\,-\,\mathcal{B^*}\overline{\mu} \,-\,
\mathbf{C}\overline{\nu}\,\right)^{-1} \varphi \,\right\rangle \\
&= \left\langle \left(\,\Id\,-\,\mu\mathcal{B}\,-\,\nu
\mathbf{C} \right)^{-1} (G\,\partial \mu\,+\,\overline{G}\,\partial\nu\,+\,\partial H) \,,\, \varphi \,\right\rangle ,
\end{aligned}
$$
for all $\varphi\in L^{s'}$ and we finish the proof of Theorem \ref{main}.

\begin{lem}\label{inverL2}
Let $\mu, \nu \in W^{1,2}(\C)$ be  compactly supported with $\| |\mu| + |\nu| \|_{\infty}\leq \frac{K-1}{K+1}$ for some $K\ge 1$. 
Then the generalized Beltrami operator
$$
 \Id-\mu\,\B  -\nu \overline{\B}:{W}^{1, q}(\C)\to {W}^{1, q}(\C)
$$
is bounded and boundedly invertible for any $1<q<2$. 
\end{lem}

The above Lemma is implicit in \cite[Proposition 6]{BCO}. The proof follows Iwaniec's scheme showing that this
Beltrami operator is injective and a Fredholm operator on $W^{1,q}(\C)$ with index 0.
Even so, we provide another proof.

\begin{proof}
Set $T:=  \Id-\mu\,\B  -\nu \overline{\B}$. Since $W^{1,2}\cap L^{\infty}$ is a multiplier of $W^{1,q}$, $1<q<2$, that is
$$
\| g\, f\|_{W^{1,q}}\le \| g\|_{\infty}\| f\|_q + \| g\|_{\infty}\| \nabla f\|_q + \| \nabla g\|_{2}\| f\|_{q*} 
\le C (\| g\|_{\infty} + \| \nabla g\|_{2}) \| f\|_{W^{1,q}} ,
$$
it is clear that $T$ is a linear bounded operator from $W^{1,q}$ to itself. Recall that $\mu ,\nu\in \operatorname{VMO}$ and $W^{1,q}(\C)
\subset L^{\frac{2q}{2-q}}(\C)$. Thus $T$ is injective on $W^{1,q}$. Also, given $g\in W^{1,q}$ there is $f\in  L^{\frac{2q}{2-q}}$ such that
$Tf=g$. We want to check that really $f\in W^{1,q}$ and therefore $T$ is also surjective on $W^{1,q}$, finishing the proof.

Fisrt we will see that the distribution $\partial f$ belongs to $L^q(\C)$.
For any $\varphi\in C^\infty_c(\C)$ we have
$$\begin{aligned}
\left\langle \partial g, \varphi\right\rangle & = -\left\langle g,\overline{\partial}\varphi\right\rangle=-\left\langle f- 
\mu\mathcal{B}f-\nu\overline{\mathcal{B}}f,\overline{\partial} \varphi\right\rangle\\
&=-\left\langle f,\overline{\partial}\varphi\right\rangle + \left\langle \mathcal{B}f,
\overline{\mu}\overline{\partial}\varphi\right\rangle + \left\langle \overline{\mathcal{B}}f,\overline{\nu}\overline{\partial}\varphi\right\rangle\\
&= -\left\langle f,\overline{\partial}\varphi\right\rangle + \left\langle \mathcal{B}f,\overline{\partial}\left(\overline{\mu}\varphi\right) \right\rangle 
+ \left\langle \overline{\mathcal{B}}f,\overline{\partial}\left(\overline{\nu}\varphi \right)\right\rangle
- \left\langle\mathcal{B}f,\overline{\partial \mu}\varphi\right\rangle - \left\langle \overline{\mathcal{B}}f,\overline{\partial \nu}\varphi\right\rangle\,.
\end{aligned}$$
If we act properly, we can see that
$$\left\langle\partial g,\varphi\right\rangle + \left\langle\mathcal{B}f,\overline{\partial \mu}\varphi\right\rangle 
+ \left\langle \overline{\mathcal{B}}f,\overline{\partial \nu}\varphi\right\rangle = -\left\langle f,\overline{\partial}\varphi\right\rangle 
+ \left\langle \mathcal{B}f,\overline{\partial}\left(\overline{\mu}\varphi\right) \right\rangle 
+ \left\langle \overline{\mathcal{B}}f,\overline{\partial}\left(\overline{\nu}\varphi \right)\right\rangle\,. 
$$
Simplifying on the left hand side and taking adjoint operators on the right hand side we can rewrite the previous equality as
\begin{equation}\label{jojojo 0}
\left\langle\partial g+ \partial\mu\mathcal{B}f+\partial\nu\overline{\mathcal{B}}f\,,\,\varphi\right\rangle = 
\left\langle\partial f\,,\,\left( \Id-\mathcal{B^*}\overline{\mu}-\mathbf{C}\overline{\nu}\right)\,\varphi\right\rangle\,
\end{equation}
where in the last term we have used $\partial\overline{\mathcal{B}}=\mathbf{C}\partial$. 
Moreover, $\partial g+ \partial\mu\mathcal{B}f+\partial\nu\overline{\mathcal{B}}f$ belongs to $L^q(\C)$ and consequently equality \eqref{jojojo 0} extends to all 
$\varphi\in L^{q'}(\C)$, $q'=\frac{q}{q-1}$. Therefore, we can take as a test function 
$\varphi=\left(\Id- \mathcal{B^*}\overline{\mu}-\mathbf{C}\overline{\nu}\right)^{-1} \psi$ with $\psi \in C^\infty_c\left(\C\right)$. 
In this way we see that for all $\psi\in C^\infty_c\left(\C\right)$, we get
\begin{equation}\label{jojojo 1}
\left\langle\partial f\,,\,\psi\right\rangle = \left\langle \left(\Id-\mu\mathcal{B}-\nu\mathbf{C}\right)^{-1}
\left[ \partial g+ \partial\mu\mathcal{B}f+\partial\nu\overline{\mathcal{B}}f \right]\,,\psi\right\rangle\, ,
\end{equation}
which implies that $\partial f \in L^q(\C)$.
To obtain that $\overline{\partial}f$ also belongs to $L^q(\C)$ we proceed as follows.
$$
\begin{aligned}
\left\langle \overline{\partial}g,\varphi\right\rangle &= - \left\langle g ,\partial \varphi\right\rangle = -\left\langle f- \mu\mathcal{B}f-\nu\overline{\mathcal{B}}f,\partial\varphi\right\rangle\\
&=\left\langle\overline{\partial} f,\varphi\right\rangle - \left\langle \overline{\partial}\left(\mu\mathcal{B}f\right),\varphi\right\rangle - \left\langle \overline{\partial}\left(\nu\overline{\mathcal{B}}f\right),\varphi\right\rangle\,.
\end{aligned}
$$
Regrouping adequately and using that $\overline{\partial}\mathcal{B} f=\partial f$ and 
$\overline{\partial}\, \overline{\mathcal{B}} f= \overline{\partial}\,\mathcal{B^*}\bar f = \mathcal{B^*}(\overline{\partial f})$, we have
$$
\begin{aligned}
\left\langle\overline{\partial}f,\varphi\right\rangle  & = \left\langle\overline{\partial}g,\varphi\right\rangle 
+ \left\langle \mu\partial f + \overline{\partial}\mu\mathcal{B}f,\varphi\right\rangle + \left\langle \nu  \mathcal{B^*}(\overline{\partial f}) 
+ \overline{\partial}\nu \overline{\mathcal{B}}f,\varphi \right\rangle \\
& = \left\langle\overline{\partial}g + \mu\partial f + \overline{\partial}\mu\mathcal{B}f +  \nu  \mathcal{B^*}(\overline{\partial f}) 
+ \overline{\partial}\nu \overline{\mathcal{B}}\, , \, \varphi\right\rangle .
\end{aligned}
$$
Consequently $\overline{\partial}f= \overline{\partial}g + \mu\partial f + \overline{\partial}\mu\mathcal{B}f +  \nu  \mathcal{B^*}(\overline{\partial f}) 
+ \overline{\partial}\nu \overline{\mathcal{B}}$ belongs to $L^q(\C)$.

\end{proof}

\section{Proof of Proposition \ref{QC}}

\noindent
In order to proceed with the proof of Proposition \ref{QC}, we start with a Lemma.

\begin{lem}\label{claim}
Let $\mu,\nu\in C^\infty(\C)$ be compactly supported and such that $\||\mu|+|\nu|\|_\infty\leq \frac{K-1}{K+1}$ for some $K\geq 1$. Let $f:\C\to\C$ be a quasiconformal solution of
\begin{equation}\label{gener}
\overline\partial f -\mu\,\partial f - \nu\,\overline{\partial f}=0.
\end{equation}
Then the inequality
$$\|D (\log \partial f)\|_{L^p(\C)}\leq C(K,p)\,\left(\||\partial \mu|+|\partial\nu|\|_{L^p(\C)}\right)$$
holds if $\frac{2K}{K+1}<p<\frac{2K}{K-1}$.
\end{lem}
\begin{proof}
Let us first remind that the complex logarithm $g=\log\partial f$ is well defined. Indeed, in the present smooth setting, $\partial f$ is a non-vanishing, complex valued smooth function, and therefore it has a well defined logarithm, wich is also smooth. We are thus legitimate to take derivatives at \eqref{gener}, and obtain
$$
\overline\partial (\partial f)-\mu\,\partial(\partial f)-\nu\partial\overline{(\partial f)}=\partial \mu\,\partial f+\partial\nu\,\overline{\partial f}
$$
whence
$$
(\Id-\nu\,\mathbf{C})(\overline\partial\partial f)-\mu\,\partial\partial f=\partial \mu\,\partial f+\partial\nu\,\overline{\partial f}
$$
Since $(\Id-\nu\,\mathbf{C})^{-1}=\frac1{1-|\nu|^2}\,(\Id+\nu\,\mathbf{C})$, we get
$$
\overline\partial \partial f -\mu_0\,\partial\partial f-\nu_0\,\overline{\partial\partial f}=\frac1{1-|\nu|^2}\,(\Id+\nu\,\mathbf{C})(\partial \mu\,\partial f+\partial\nu\,\overline{\partial f})
$$
with
$$\aligned
\mu_0&=\frac{\mu}{1-|\nu|^2}\\\nu_0&=\frac{\nu\overline\mu}{1-|\nu|^2}\endaligned$$
Now, we recall that
$$
\aligned
\overline\partial g&=\frac{\overline\partial \partial f}{\partial f}\\
\partial g&=\frac{\partial \partial f}{\partial f}\\
\overline{\partial g}&=\frac{\overline{\partial \partial f}}{\overline{\partial f}}
\endaligned
$$
Thus
$$
\overline\partial g -\mu_0\,\partial g-\tilde{\nu}_0\,\overline{\partial g}=\frac1{\partial f}\,\frac1{1-|\nu|^2}\,(\Id+\nu\,\mathbf{C})(\partial \mu\,\partial f+\partial\nu\,\overline{\partial f})
$$
where $|\tilde{\nu}_0|=|\nu_0|$. Now, since
$$
\left|\frac1{\partial f}\,\frac1{1-|\nu|^2}\,(\Id+\nu\,\mathbf{C})(\partial \mu\,\partial f+\partial\nu\,\overline{\partial f})\right|\leq C(K)\,(|\partial\mu|+|\partial\nu|)
$$
we can use \cite{AIS}(e.g. \cite[Theorem 14.0.4]{AIM}) to state that
$$\|Dg\|_{L^p(\C)}\leq C(K,p)\,\||\partial\mu|+|\partial\nu|\|_{L^p(\C)}$$
when $p$ lies in the critical interval $(\frac{2K}{K+1},\frac{2K}{K-1})$. The claim follows.
\end{proof}

\begin{proof}[Proof of Proposition \ref{QC}]
In the assumptions of Proposition \ref{QC}, $f$ is holomorphic at $\infty$. Being also bijective, it can only have linear growth as $|z|\to\infty$, 
so that there must exist $a\in\C$, $a\neq 0$, and $C>0$ such that
$|f(z)-a\,z|\leq \frac{C}{|z|}$ as $|z|\to\infty$. In particular, this tells us that the value $\partial f(\infty)$ is well defined, 
and uniquely determines a well defined branch of the complex logarithm $\log \partial f$, precisely as done in \cite{AIPS}. \\
\\
We mollify $\mu$ and $\nu$ and obtain a sequence $\mu_n,\nu_n\in C^\infty$ of compactly supported coefficients, $|\mu_n|+|\nu_n|\leq |\mu|+|\nu|\leq \frac{K-1}{K+1}$ such that
$$\lim_{n\to\infty}\|\mu_n-\mu\|_{W^{1,p}(\C)}+\|\nu_n-\nu\|_{W^{1,p}(\C)}=0.$$
To each pair $\mu_n,\nu_n$ we can associate a unique quasiconformal map $f_n$ such that $\partial f_n-a = \B(\overline\partial f_n)$, and in particular
$|f_n(z)-a\,z|\leq \frac{C}{|z|}$ as $|z|\to\infty$, since $\overline\partial f_n$ has compact support. 
It can be shown that in this situation the sequences $\{\overline\partial f_n\}_n$ and $\{\partial f_n-a\}_n$ are bounded in $L^r(\C)$ 
for each $2<r<\frac{2K}{K-1}$, and indeed, modulo subsequences,
$$\lim_{n\to\infty}\|Df_n-Df\|_{L^r(C)}=0,\hspace{2cm}2<r<\frac{2K}{K-1}.$$
We further know that $g_n=\log \partial f_n$ is well defined an moreover $\partial g_n\,\partial f_n=\partial\partial f_n $. But both $\partial f_n$ and $\partial g_n$ can be granted a degree of integrability independent of $n$. Indeed, from Lemma \ref{claim} there is a uniform bound
$$\aligned
\|Dg_n\|_{L^p(\C)}
&\leq C(K,p)\,\left(\||\partial \mu_n|+|\partial\nu_n|\|_{L^p(\C)}\right)\\
&\leq C(K,p)\,\left(\||\partial \mu|+|\partial\nu|\|_{L^p(\C)}\right)
\endaligned$$
Combining this with the optimal integrability of quasiconformal maps by Astala \cite{A} we obtain local $L^s$ bounds for $\partial\partial f_n$, independent of $n$, whenever
$$\frac1s>\frac{K-1}{2K}+\frac1p.$$
It then follows that $\{\partial\partial f_n\}_n$ has a weak accumulation point in $L^s_{loc}(\C)$, which obviously can only be $\partial\partial f$, so that $\partial\partial f\in L^s_{loc}(\C)$. A similar argument shows that $\overline\partial \partial f$ belongs to $L^s_{loc}(\C)$ (actually $L^s(\C)$, since $\overline\partial f$ has compact support). In order to see that $\overline\partial\overline\partial f\in L^s (\C)$, just notice that from $\partial f - a = \B(\overline\partial f)$ one gets $\overline\partial f=\B^\ast(\partial f-a)$. Thus,  $\overline\partial\overline\partial f =\B^\ast(\overline\partial \partial f)$. \end{proof}

\noindent {\textbf{Acknowledgements}}. The authors were partially supported by research grants 2014SGR75 (Generalitat de Catalunya), MTM2016-75390-P (Spanish government)
and FP7-607647 (European Union).

\noindent
Antonio L. Bais\'on Olmo\\
   Departamento de Ciencias B\'asicas.\\
   Universidad Aut\'onoma Metropolitana - Unidad Azcapotzalco.\\
   Av. San Pablo No. 180\\
   Col. Reynosa Tamaulipas\\
   C.P. 02200\\
   Delegación Azcapotzalco\\
   Distrito Federal (M\'exico)\\
   A. Clop, J. Orobitg\\
Departament de Matem\`atiques\\
Universitat Aut\`onoma de Barcelona\\
08193-Bellaterra (Catalonia)

\vspace{0.2cm}
\noindent{\it E-mail addresses}:\\
\texttt{albo@correo.azc.uam.mx}\\
\texttt{albertcp@mat.uab.cat}\\
\texttt{orobitg@mat.uab.cat}

\end{document}